\documentclass[letterpaper,10pt,conference,twoside]{ieeeconf}

\usepackage{amsmath,amssymb,color,cite}%,showkeys}

\usepackage{graphicx}
\usepackage{dblfloatfix}
\usepackage{latexsym}
\usepackage{algorithm}
\usepackage{algpseudocode}
\usepackage{verbatim}
\usepackage{cite}
\usepackage[latin1]{inputenc}
\usepackage{enumerate}

\usepackage{amsthm}
\usepackage{bbm}

\usepackage{wrapfig}

\usepackage{pgf}
\usepackage{pgfplots,tikz}
\usetikzlibrary{tikzmark}
\usepackage[font=footnotesize]{caption}
\usepackage{subcaption}
\usepackage{arydshln}
\usepackage{faktor}
\usepackage{amsfonts}
\usepackage{amsbsy}
\usepackage{bm}
\usepackage{booktabs}

\usepackage{tikz}
\usetikzlibrary{tikzmark,patterns}
\usetikzlibrary{shapes,arrows,arrows.meta,chains,matrix,positioning,scopes,calc,patterns,shapes.geometric,intersections,backgrounds,fit,decorations.text}
\tikzset{
block/.style = {draw, rectangle,
	minimum height=0.5cm,
	minimum width=1cm},
input/.style = {coordinate,node distance=1cm},
output/.style = {coordinate,node distance=4cm},
arrow/.style={draw, -latex,node distance=2cm},
pinstyle/.style = {pin edge={latex-, black,node distance=2cm}},
sum/.style = {draw, circle, node distance=1cm},
}

\DeclareMathOperator*{\argmin}{arg\,min}

\newcommand{\norm}[1]{\ensuremath{\| #1 \|}}

\let\leq\leqslant
\let\geq\geqslant

\newcommand{\R}{\mathbb R}
\newcommand{\N}{\mathbb N}

\newcommand{\calC}{\ensuremath{\mathcal{C}}}

\newcounter{todocounter}
\usepackage[colorinlistoftodos]{todonotes}

\newtheorem{theorem}{Theorem}[section]

\newtheorem{lemma}[theorem]{Lemma}

\theoremstyle{remark}
\newtheorem{remark}[theorem]{Remark}

\theoremstyle{definition}
\newtheorem{assumption}{Assumption}
\newtheorem{definition}[theorem]{Definition}
\newtheorem{example}{Example}

\title{On analysis of open optimization algorithms}

\author{Jaap Eising \& Florian D\"{o}rfler \thanks{The authors are with the Automatic Control Laboratory (IfA), ETH Z\"{u}rich, Switzerland (\{jeising, doerfler\}@control.ee.ethz.ch). This work is supported by the SNF/FW Weave Project 200021E\_20397}}

\begin{document}

\maketitle

\begin{abstract}
	We consider optimization algorithms that are open systems, that is, with external inputs and outputs. Such algorithms arise for instance, when analyzing the effect of noise or disturbance on an algorithm, or when an algorithm is part of control loop without timescale separation. Bridging between monotone operator theory and energy-based modeling, we consider analysis results in the form of incremental dissipativity certificates, yielding tests in the form of linear matrix inequalities. To be precise, we consider robustness in terms of incremental small gain, and composition results for optimization algorithms operating in closed loop.
\end{abstract}

\section{Introduction} 

Instead of designing controllers in closed-form, many modern implementations of controllers take the form of feedback from explicit optimization problems. While providing a full overview is impossible, some particularly popular methods with optimization explicitly in the control loop are Model Predictive Control \cite{JBR-DQM:09}, its data-based variants \cite{JC-JL-FD:21}, Online Feedback Optimization \cite{OFO20,HAUSWIRTH2024100941}, and safety filters \cite{Ames17}. Accordingly, the aforementioned papers point to a rich literature regarding the analysis of closed loops between plants and optimization-based controllers. 

In many applications, fully resolving the optimization problem at the same frequency as the plant is impossible. A common approach is to consider an iterative optimization algorithm which is guaranteed to converge to the optimal value, iterate for one or more steps, and feed the resulting suboptimal value to the plant. Assuming that for each time instance of the plant the optimization algorithm can converge to within a pre-specified distance to the optimal value, can be combined with robustness analysis to yield closed loop guarantees. 

The question arises which algorithm to use. This work stands within the context of system theory for algorithms \cite{STOA24}, that is, viewing algorithms as dynamical systems, unifying ideas from both. In this, designing an optimization algorithm boils down to designing an autonomous system which converges to the optimal control input. In practice, implementation of algorithms is imprecise or noisy, especially on embedded platforms and subject to real-time computing at the sampling rate of the system. Taking the system-theoretic view, this can be modeled as a disturbance acting on the algorithm, and foundational to robust control is the trade-off between robustness and stability. Indeed, convergence and robustness of optimization algorithms exhibit the same trade-off \cite{BvS-LL:21}. Hence, in such cases, simply picking the fastest-converging optimization algorithm can have unforeseen side effects or, in the worst case, destabilize. 

More broadly, we are interested in situations with external signals that can not necessarily be modeled as disturbances.  In order to address these issues, we propose to investigate the analysis of \textit{open} optimization algorithms, that is, optimization algorithms which are not autonomous, but have remaining inputs and outputs. Since, for closed algorithms, the natural analysis concept is contraction theory, we consider its open equivalent of incremental dissipativity in this paper. 

\subsection*{Literature review} 
The analysis of optimization algorithms as dynamical systems is a topic with a lot of history, see e.g \cite{helmke2012optimization} and the references therein. More specifically, the use of dissipativity theory (\cite{JCW:72a,JCW:72b}) for such analysis is discussed in\cite{Lessard-22}. Essentially this method takes the following approach. We decompose an algorithm into a linear system which is in interconnection with a \textit{sector bounded} oracle, in the simplest case taking the form of a gradient evaluation. Then we can view the convergence of the algorithm as a variant on the classical ``problem of absolute stability'' or Lur'e problem \cite{lurie1944theory}. 

Beyond the use of this classical notion of dissipativity, a number of related concepts are applied to derive different analysis results for optimization algorithms. The paper \cite{SP-19} uses \textit{equilibrium independent dissipativity}, to avoid the requirement of having (for analysis only) the optimal point at the origin. On the other hand, integral quadratic constraints, such as in \cite{Lessard16,Scherer-21,scherer23} can be used to guarantee similar conditions without constructing a storage function.

One variant on dissipativity theory arose based on the fact that many nonlinear systems seem better suited to incremental analysis than to stability analysis \cite{Zames60,LOHMILLER1998683}. Contraction theory (see e.g. \cite{FB-CTDS}) deals with increments (differences) of trajectories instead of with trajectories directly. In this sense, it can be viewed as the \textit{incremental} form of Lyapunov theory. Similar to how dissipativity theory is the generalization of Lyapunov theory to open systems, incremental dissipativity generalizes contraction in the same way \cite{SEPULCHRE2022290}. Of particular relevance to this work are the clear links between incremental dissipativity and monotone operator theory \cite{bauschke2011, ryu2016primer}, the cornerstone of modern algorithm analysis.

%While we only deal with discrete algorithms, there is also work in the analysis of gradient flows, which are the continuous-time equivalent of our optimization algorithms. For an overview, we point to \cite{Ahmed24}  and the references therein. Some recent works warrant additional mention. For instance \cite{fu2023novel} considers control algorithms in a framework of \textit{Krasovskii} passivity, a form of differential dissipativity. On the other hand, \cite{ad-vc-ag-gr-fb:23f} illustrates the power of contraction theory in handling time-varying problems.\todo{The CT part might well be removed}

\subsection*{Contribution}

We investigate incremental dissipativity as a tool to analyze optimization algorithms which can be written as the interconnection of a linear system and a set of oracles. In this work
\begin{enumerate}
	\item we introduce a class of \textit{open} optimization algorithms and show its relevance;
	\item we derive linear matrix inequality-based test for robustness analysis of open optimization algorithms, yielding incremental small gain guarantees with respect to the propagation of noise or disturbances; 
	\item we provide results allowing us to analyze closed-loop performance of optimization-in-the-loop, that is, the direct interconnection of an algorithm and a plant.
\end{enumerate}

\subsection*{Notation}
We denote by $\N$ and $\R$ the sets of non-negative
integer and real numbers, respectively. We let $\R^{n\times m}$
denote the space of $n\times m$ real matrices. For any
$x\in\R^n$, $\Vert x\Vert_2$ denotes the standard 2-norm. 
If $P=P^\top\in\R^{n\times n}$, then $P\succeq 0$ (resp. $P\succ 0$) denotes
that $P$ is positive semi-definite (resp. definite). A set $\calC\subseteq\mathbb{R}^n$ is convex if $\alpha x +(1-\alpha) y\in \calC$ for $x,y\in\calC$ and $0\leq\alpha\leq 1$. Let $\mu>0$, a function $c:\calC\to \mathbb{R}$ is called $\mu$-strongly convex if, for all $x_1,x_2\in\calC$ and $0\leq\alpha\leq 1$, we have
\begin{align*}
	f(\alpha x_1 + (1-\alpha) x_2) & \leq \alpha f(x_1) + (1-\alpha) f(x_2)\\
	& \quad -\alpha(1-\alpha)\frac{\mu}{2} \norm{x_1-x_2}_2^2.
\end{align*}
The function $c$ is convex if the latter holds with $\mu=0$.

\section{Open optimization algorithms} 

In line with \cite{STOA24} we consider optimization algorithms that can written as a dynamical systems. More specifically, we consider algorithms as in \cite{Lessard-22}, that is, those decomposed into linear dynamic system and a set of \textit{oracles}, such as gradient evaluations or projections. In this paper, we treat the oracles as given, and outside of our consideration. To motivate this, consider the following example:
\begin{example}[Nesterov acceleration] 
	\label{ex:closed loop}
	Consider the problem of minimizing a function $c:\mathbb{R}^q\rightarrow \mathbb{R}$. Assume that $c$ is continuously differentiable and that we have access to an oracle which calculates the gradient $\nabla c$. A popular method of resolving this problem is Nesterov's accelerated method \cite{Nesterov-book}. As shown in \cite{Lessard-22}, we can write method this as a linear system: 
	\begin{align*} 
		x_{k+1} &= \begin{bmatrix} 1+\beta & -\beta \\ 1 &0 \end{bmatrix} x_k + \begin{bmatrix} -\eta \\ 0 \end{bmatrix} u_k, \\
		y_k &= \begin{bmatrix}1+\beta & -\beta\end{bmatrix}x_k,
	\end{align*}
	in interconnection with the $ u_k = \nabla c(y_k)$. \hfill $\blacksquare$
\end{example}

In a similar fashion, we can address methods like projected gradient descent and the Alternating Direction Method of Multipliers (ADMM). For this, we require different oracles, beyond gradient operations. Some particularly relevant choices are subgradients, the Euclidean projection onto a (convex) set $C$,
\[ \Pi_C (y) = \argmin_{u\in C} \norm{u-y}^2_2, \] 
or the \emph{proximal operator} of $c$,
\[ \textup{prox}_c (y) = \argmin_{u} c(u)+\tfrac{1}{2} \norm{u-y}^2_2. \]

In contrast with the autonomous, or \textit{closed}, optimization algorithms above, in this paper we are interested in optimization algorithms with additional inputs and outputs. To be precise, optimization algorithms of the form 
\begin{subequations}\label{eq:linear system}
	\begin{align}  
		x_{k+1} &= A\hspace{4pt}x_k+B_1\hspace{4pt}u_k+B_2\hspace{4pt}d_k, \\
		y_k &= C_1x_k+D_{11}u_k +D_{12}d_k, \\
		z_k &= C_2x_k+D_{21}u_k +D_{22}d_k.
	\end{align}
\end{subequations}
Here, $x \in \mathbb{R}^n$ is the \textit{state} of the algorithm, $u\in\mathbb{R}^{m_u}$ is the \textit{oracle input}, $y\in \mathbb{R}^{p_y}$ is the \textit{oracle output}, $d_k\in\mathbb{R}^{m_d}$ is the \textit{external input}, and $z_k\in\mathbb{R}^{p_z}$ is the \textit{external output}. As indicated by the naming above, we also consider \textit{oracles} of the form:
\begin{equation} 
	\label{eq:general oracles} u_k = \phi(y_k).
\end{equation}
%where, if the oracle is independent of external variables, for the dependence on $d_k$ is suppressed. 

\begin{figure}
	\begin{center}
	\begin{tikzpicture}[auto, node distance=1cm,>=latex']
		\node [block, name=plant,minimum height=1.2cm,minimum width=2.4cm] (plant) {$\Sigma$};
		\node [block, left= 1 of plant ,minimum height=1.2cm,minimum width=2.4cm] (K1) {$u = \varphi(y)$};
		
		\draw [draw,->] ($(plant.east)+(0,0.2)$) -- node {$z$} ($(plant.east)+(1,0.2)$);
		\draw [draw,->] ($(plant.east)+(1,-0.2)$) -- node {$d$} ($(plant.east)+(0,-0.2)$);
		\draw [draw,->] ($(plant.west)-(0,0.2)$) -- node {$y$} ($(plant.west)-(1,0.2)$);
		\draw [draw,->] ($(plant.west)-(1,-0.2)$) -- node {$u$} ($(plant.west)-(0,-0.2)$);
	\end{tikzpicture}
	\end{center}
	\caption{A schematic view of an open optimization algorithm. Here, $\Sigma$ denotes the linear system \eqref{eq:linear system}. } \label{fig:open opt}
\end{figure}

We consider \textit{open optimization algorithms} consisting of a system \eqref{eq:linear system} interconnected with \eqref{eq:general oracles} (cf. Figure~\ref{fig:open opt}). This means that the algorithm still has external inputs $d_k$, such as disturbances, and outputs $z_k$, such as measurements for a performance metric.

More specifically, we consider the analysis of open optimization algorithms. To derive guarantees, it is clearly insufficient to investigate just the internal convergence rates of the algorithm. In the language of control design, instead of considering internal stability of a closed system, we need to consider the input-output behavior of an \emph{open} system. In this paper, we will employ \textit{incremental dissipativity theory} as the natural generalization of contraction analysis for the analysis of input-output behavior of class of optimization algorithms. In particular, we will aim at analyzing algorithms with \textit{robustness} and \textit{closed loop behavior} in mind. 

\section{Examples of open optimization algorithms}
Before starting with the analysis, we will show how a number of relevant problems take the form described. 

\subsection{Closed-loop optimization}
%
%Consider the optimization problem 
%\begin{equation}\label{eq:opt cont 1}
%\min_x c(x),
%\end{equation}  
%where $c:\mathbb{R}^q\rightarrow\mathbb{R}^q$. We are interested in Writing an optimization algorithm as an autonomous dynamical system, 
%\[ u_{k+1} = h(u_k), \] 
%means that convergence analysis of the algorithm is equivalent to contraction analysis of the dynamical system. 
%
Consider a discrete time plant 
\begin{equation}\label{eq:plant} \xi_{k+1} = f(\xi_k,\nu_k), \quad  \zeta_k = g(\xi_k,\nu_k),\end{equation}
where the state signal $\xi_k\in\mathbb{R}^n$, the input $\nu_k\in \mathbb{R}^m$, and the output $y_k\in \mathbb{R}^p$, and hence the functions $f:\mathbb{R}^n\times \mathbb{R}^m \rightarrow \mathbb{R}^n$ and $g:\mathbb{R}^n\times \mathbb{R}^m \rightarrow \mathbb{R}^p$. Now consider a controller of the form:
\begin{equation}\label{eq:opt cont} \nu_k = \argmin_{\nu} c(\nu,\zeta_k), \end{equation}
where $c: \mathbb{R}^m\times \mathbb{R}^p\rightarrow \mathbb{R}$ denotes a \textit{cost function}. Controllers of this form mainly, but not exclusively, arise in the form of Model Predictive Control (MPC). Instead of investigating the idealized closed loop of \eqref{eq:plant} and \eqref{eq:opt cont}, we will consider the situation where an iterative optimization algorithm is used to solve \eqref{eq:opt cont}, which runs in the control loop with the same sampling rate. The following example will illustrate some of the arising difficulties.  

\begin{example}[Illustrative example] 
	Consider the closed loop between a linear scalar system and a controller optimizing a single-stage performance index, given by: 
	\[\xi_{k+1} = -\xi_k + \nu_k, \quad \nu_k =\argmin_\nu \tfrac{1}{2} \norm{\nu- K\xi_k}^2.\] 
	Clearly, the closed loop is asymptotically stable if and only if $0<K<2$. 
	Now we replace the optimization algorithm with a simple gradient descent scheme:
	\[\nu_{k+1} = \nu_k - \eta (\nu_k -K\xi_k)= (1-\eta)\nu_k + \eta K\xi_k.\]
	When $\xi_k$ is considered to be constant, this algorithm is stable for $0<\eta<2$. 
	However, writing the dynamics of the closed loop, we obtain 
	\[ \begin{bmatrix} \xi_{k+1} \\ \nu_{k+1} \end{bmatrix}= \begin{bmatrix} -1 & 1 \\ \eta K & 1-\eta\end{bmatrix} \begin{bmatrix} \xi_k \\ \nu_k \end{bmatrix}.\]
	Simple calculation reveals that the closed loop is unstable for \textit{any} pair $0<K,\eta<2$. On the other hand, for $K=3$ and $\eta=-\tfrac{1}{2}$, we do obtain a stable closed loop, despite the fact that neither the idealized closed loop, nor the algorithm is contracting.  \hfill $\blacksquare$
\end{example} 
This motivates our first problem: Given an open algorithm of the form \eqref{eq:linear system} interconnected with the oracle \eqref{eq:general oracles} and a plant of the form \eqref{eq:plant}, provide conditions on the plant and the algorithm under which the closed loop using the equalities $\zeta_k =d_k$ and $\nu_k= z_k$ is contractive.

\subsection{Robustness of algorithms} 
A second problem considers disturbances on optimization algorithms.

\begin{example}[Open Nesterov with gradient noise] 
	\label{ex:open nesterov 1}
	Consider again Example~\ref{ex:closed loop}. Now assume that we do not have access to perfect measurements of the gradient, but that these are corrupted by noise or other uncertainties, that is, $\bar{u}_k = d_k + \nabla c(y_k)$. We can write:
	\begin{align*} 
		x_{k+1} &= \begin{bmatrix} 1+\beta & -\beta \\ 1 &0 \end{bmatrix} x_k + \begin{bmatrix} -\eta \\ 0 \end{bmatrix} u_k + \begin{bmatrix} -\eta \\ 0 \end{bmatrix} d_k , \\
		z_k &= y_k = \begin{bmatrix}1+\beta & -\beta\end{bmatrix}x_k,
	\end{align*}
	in interconnection with the oracle $u_k = \nabla c(y_k)$.  \hfill $\blacksquare$
	%	A block diagram of this situation is seen in Figure~\ref{fig:closed loop grad noise}.
\end{example}

%\begin{example}[Open Nesterov with measurement noise] 
%	\label{ex:open nesterov 2}	
%	Alternatively, if we can not precisely calculate the gradient at $\bar{y}_k$, but instead obtain $ u_k = \nabla c(\bar{y}_k+d_k)$, 	we can equivalently write:
%	\begin{align*} 
	%		x_{k+1} &= \begin{bmatrix} 1+\beta & -\beta \\ 1 &0 \end{bmatrix} x_k + \begin{bmatrix} -\eta \\ 0 \end{bmatrix} u_k, \\
	%		y_k &= \begin{bmatrix}1+\beta & -\beta\end{bmatrix}x_k + d_k, \\
	%		z_k &= \begin{bmatrix}1+\beta & -\beta\end{bmatrix}x_k. 
	%	\end{align*}
%	in interconnection with the oracle $u_k = \nabla c(y_k)$.
%\end{example}

Note that this is an open optimization algorithm of the form \eqref{eq:linear system} interconnected with the oracle \eqref{eq:general oracles}. If $d_k=0$ there is no noise, and analysis of the algorithm can proceed by known methods. However, as with dynamical systems, the noiseless behavior does not provide guarantees in the presence of noise. Hence, the second problem is to analyze the robustness trade-off of open optimization algorithms to noise.

%\todoinr{Remove or rewrite and come back here} 
%\subsection{Bilevel optimization} 
%
%
%	Consider the problem 
%	\[ \min_x f(x,y) \quad \textrm{ subject to } y= \argmin_{y} g(y) \] 
%	We can relax this problem by in the way we do above. We consider two interconnected optimization algorithms. One for $y$, with an output equal to $y$, e.g. 
%	\[ y_{k+1} = y_k - \nabla g( y_k).\]
%	Note that if e.g. $g$ is stricty convex, then we can design the linear part in the way above. A second optimization algorithm is to $x$, with a fixed value of $\bar{y}$, e.g. 
%	\[ x_{k+1} = x_k - \nabla_x f(x_k , \bar{y}). \]
%
%\subsection{Time-scale separation} 

\section{Technical preliminaries}\label{sec:tech}
As a first step towards the analysis of open optimization algorithms, we will consider relevant properties of the oracles. 

\subsection{Oracle properties} \label{ssec:oracles}

Let $\varphi:\mathbb{R}^q\rightarrow\mathbb{R}^p$, then the function $\varphi$ is said to be $L$ \emph{Lipschitz} if:
\[ \norm{\varphi(x_1)-\varphi(x_2)}^2 \leq L^2 \norm{x_1-x_2}^2,\] 
for all $x_1,x_2\in\mathbb{R}^n$. If additionally $p=q$, then we say $\varphi$ is 
\begin{itemize}
	\item $\mu-$\emph{strongly monotone} if
	\[ (\varphi(x_1)-\varphi (x_2))^\top(x_1-x_2)\geq \mu \norm{x_1-x_2}_2^2, \]
	\item \emph{firmly nonexpansive} if
	\[ (\varphi(x_1)-\varphi(x_2))^\top(x_1-x_2) \geq \norm{\varphi(x_1)-\varphi(x_2)}_2^2,\]
\end{itemize}
for all $x_1,x_2\in\mathbb{R}^n$. 

The following lemma links properties of commonly used oracles to assumptions on their related optimization problem. 
\begin{lemma} Let $c:\mathbb{R}^n\rightarrow \mathbb{R}$.
	\begin{enumerate}
		\item If the function $c$ is differentiable then it is $\mu$-strongly convex if and only if $\nabla c$ is $\mu-$\emph{strongly monotone}.
		\item If $c$ is convex, then the proximal operator $\textup{prox}_c$ is firmly nonexpansive.
		\item If $C\subseteq\mathbb{R}^n$ is a convex set, then $\Pi_C$, is firmly nonexpansive.
	\end{enumerate}
\end{lemma}

One key observation is that all of these properties are \textit{incremental}, that is, they relate the change in function values to changes in the argument. To emphasize this, we will use shorthand notation for incremental signals, by writing $\Delta x$ where $\Delta x_k := x^1_k-x^2_k$ for any two signals $x^1,x^2$ of the same system. 

Now, if we consider $u_k = \varphi(y_k)$, we will focus on \emph{quadratic} bounds of the increments of the form 
\begin{equation}\label{eq:sector bound}  s(\Delta y_k, \Delta u_k) \leq 0, \end{equation}
with
\begin{equation}\label{eq:quad supply}  s(\Delta y_k, \Delta u_k) := \begin{pmatrix} \Delta y_k \\ \Delta u_k\end{pmatrix}^\top S \begin{pmatrix} \Delta y_k \\ \Delta u_k\end{pmatrix}, \end{equation}
where $S=S^\top \in\mathbb{R}^{(p+m)\times (p+q)}$. As an example, we can see that $\varphi$ is $\mu$-strongly monotone if and only if: 
\begin{equation} \label{eq:mu str mon in quad form} \begin{pmatrix} \Delta y_k \\ \Delta u_k\end{pmatrix}^\top \begin{pmatrix} 2\mu I_q& -I_q \\ -I_q & 0 \end{pmatrix} \begin{pmatrix} \Delta y_k \\ \Delta u_k\end{pmatrix}\leq 0.
\end{equation}
Conditions similar to \eqref{eq:mu str mon in quad form}  can be written for firmly nonexpansive or $L$-Lipschitz functions. Moreover, while outside the scope of this paper, \emph{one-sided Lipschitz} and \emph{cocoercive} functions admit similar descriptions. For the following analyis, we will refer to the function $s$ as an \textit{(incremental) supply rate}.

\subsection{Incremental dissipativity and contraction} 
The assumptions on the oracles, as described above, pertain to static maps. For the analysis of optimization algorithms, which are dynamical systems, we will consider \textit{incremental dissipativity}, the key concept used for analysis in this paper.
\begin{subequations}
\begin{definition}
	A dynamical system with states $x\in\mathbb{R}^n$, inputs $u\in\mathbb{R}^m$, and outputs $y\in\mathbb{R}^p$ is \emph{incrementally dissipative} with respect to the \emph{supply rate} $s:\mathbb{R}^p\times\mathbb{R}^m\rightarrow \mathbb{R}$ if there exists a nonnegative \textit{storage function} $V:\mathbb{R}^n \rightarrow \mathbb{R}$, such that for any time $k$:

	\begin{equation}\label{eq:dissipation} V(\Delta x_{k+1}) - V(\Delta x_k) \leq s(\Delta y_k, \Delta u_k), \end{equation}
	for any incremental trajectory $(\Delta x,\Delta u,\Delta y)$ of the system. 
\end{definition}

For linear systems and quadratic supply rates, it is well-known (cf. \cite{JCW:72b}) that the system is dissipative if and only if it is dissipative with a quadratic storage function. By the same methods, the same follows for the incremental case. We can hence test linear systems for dissipativity using Linear Matrix Inequalities (LMI's) as follows. 
\begin{lemma}
	\label{lem:LMI}
	The linear system \eqref{eq:linear system} is incrementally dissipative with respect to the supply rate \eqref{eq:quad supply} and storage $V(\Delta x_k)=(\Delta x_k)^\top P (\Delta x_k)$ with $P\succeq 0$ if 
	\[ \begin{bmatrix} I & 0 \\ A & B\end{bmatrix}^{\!\top}\!\! \begin{bmatrix} P & 0 \\ 0 & \!\!-P \end{bmatrix}\! \begin{bmatrix} I & 0 \\ A & B\end{bmatrix} + \begin{bmatrix} C & D \\ 0 & I\end{bmatrix}^{\!\top}\!\! S \!\begin{bmatrix} C & D \\ 0 & I\end{bmatrix} \!\succeq 0.\] 
\end{lemma}	

We will generally make a number of slightly more restrictive assumptions. First, we will focus on the situation where $P$ is positive definite\footnote{If the pair $(C,A)$ is observable and the matrix $S$ has at least $p$ negative eigenvalues, then there exists a positive semi-definite quadratic storage function if and only if there exists a positive definite one \cite[Lemma 4.4] {Azkathesis23}.}. For technical reasons we will often consider instead of \eqref{eq:dissipation}, more \textit{strict} incremental dissipation inequalities: 
\begin{align}   
	V(\Delta x_{k+1}) - V(\Delta x_k) &<s(\Delta y_k, \Delta u_k), \textup{ or } \label{eq:dissipation-strict} \\
	V(\Delta x_{k+1}) - \gamma V(\Delta x_k) &\leq s(\Delta y_k, \Delta u_k), \label{eq:dissipation-exp}
\end{align}
\end{subequations}
with $0<\gamma<1$. Lemma~\ref{lem:LMI} can be adapted to these problems in a routine manner.

As a last special case, we consider the case where $s(\Delta y_k, \Delta u_k)\leq 0$ for all increments $\Delta y_k$ and $\Delta u_k$. Then, we can employ the function $V(\Delta x_k)=(\Delta x_k)^\top P (\Delta x_k)$ with $P\succ 0$ as an \emph{incremental Lyapunov function} or \emph{contraction metric}. To avoid technicalities (cf. \cite{FB-CTDS}), if for any incremental trajectory $\Delta x$ of a system, we have
\begin{subequations}
	\begin{align}   
		V(\Delta x_{k+1}) - V(\Delta x_k) &\leq 0, \textup{ or }  \label{eq: nonexp} \\
		V(\Delta x_{k+1}) - V(\Delta x_k) &<0, \textup{ or } \label{eq: contr}  \\
		V(\Delta x_{k+1}) - \gamma V(\Delta x_k) &\leq 0 \textup{ with }  0<\gamma<1, \label{eq:stric contraction}
	\end{align}
\end{subequations}
for all $k\geq 0$, we say that the system is non-expansive, contracting, or exponentially contracting (sometimes referred to as \emph{linear convergence}) with rate $\gamma$ respectively. Clearly, (exponentially) contracting systems converge to a fixed point.

\section{Analysis of open optimization algorithms} 

We return our focus to the analysis of open optimization algorithms, that is, properties of the interconnection between \eqref{eq:linear system} and \eqref{eq:general oracles}. Our goal is to verify whether the interconnction is incrementally dissipative with respect to a supply rate

\[s_{\textup{ext}}(\Delta z_k, \Delta d_k) := \begin{pmatrix} \Delta z_k \\ \Delta d_k\end{pmatrix}^\top S_{\textup{ext}} \begin{pmatrix} \Delta z_k \\ \Delta d_k\end{pmatrix}.\] 

In order to test an open optimization algorithm for such properties, we use the following line of reasoning: 

\begin{enumerate}
	\item We write the optimization algorithm as the closed loop between a linear system \eqref{eq:linear system} and a set of oracles \eqref{eq:general oracles}.  
	\item\label{item:oracle} We assume (cf. Section~\ref{ssec:oracles}) that the oracle satisfies a quadratic incremental bound of the form
	\begin{equation}\label{eq:Sphi} s_\varphi(\Delta y,\Delta u) := \begin{pmatrix} \Delta y_k \\ \Delta u_k\end{pmatrix}^\top S_\varphi \begin{pmatrix} \Delta y_k \\ \Delta u_k\end{pmatrix}\leq 0.\end{equation} 
	\item\label{item:diss} We use Lemma~\ref{lem:LMI} to show whether there exists a $\lambda>0$ for which the linear system is incrementally dissipative with respect to the supply rate
	\[ \lambda s_\varphi+ s_{\textup{ext}}, \] 
	with some storage function induced by $P\succ 0$.
	\item If so, we combine \ref{item:oracle} and \ref{item:diss} and conclude that, for the open optimization algorithm, we have
	\begin{equation}\label{eq:diss open opt}  V(\Delta x_{k+1}) - V(\Delta x_k) \leq s_{\textup{ext}}(\Delta z_k, \Delta d_k) . \end{equation} 
\end{enumerate}

Of course, we can derive similar lines of reasoning for formulations along the line of \eqref{eq:dissipation-strict} or \eqref{eq:dissipation-exp}. 

Now that we can test for the incremental dissipativity of open optimization algorithms, we will show how this helps resolve the presented problems. 

%\todoinr{Do we want to talk about overshoot?}
%\begin{corollary} Suppose that $\Delta x$ is such that \eqref{eq:stric contraction} holds for all $k\geq 0$, with $V(\Delta x_k) = (\Delta x_k)^\top P (\Delta x_k)$ and $P\succ 0$. Let $\lambda_-(P)$ and $\lambda_+(P)$ denote the smallest and largest eigenvalue of $P$ respectively, then we can conclude that 
%	\[ \norm{\Delta x_k}^2_2 \leq \gamma^k\frac{\lambda_+(P)}{\lambda_-(P)} \norm{\Delta x_0}^2_2,\]
%	for any $k\geq 0$.
%\end{corollary}
%As a consequence, if $x^*$ is a stationary point of the dynamics, then this condition yields 
%\[ \norm{x_k-x^*}^2_2 \leq \gamma^k\frac{\lambda_+(P)}{\lambda_-(P)} \norm{x_0-x^*}^2_2.\]
%Hence, we can conclude exponential contraction towards $x^*$, and quantify the maximal overshoot in terms of the eigenvalues of $P$. 

\subsection{Analysis of closed loop algorithms} 
Consider, again, a discrete time plant of the form \eqref{eq:plant}. Recall that we aim at replacing the controller \eqref{eq:opt cont} with a dynamical system. When replacing the controller with an algorithm, we obtain a autonomous closed loop consisting of the nonlinear system \eqref{eq:plant}, the linear part of the algorithm \eqref{eq:linear system}, the set of oracles \eqref{eq:general oracles}, and the interconnection equalities $\zeta_k =d_k$ and $\nu_k= z_k$.

In line with the previous, we will take an approach based on compositional properties of incremental dissipativity. Hence, we assume that the plant is incrementally dissipative. In the linear case this could be verified using a LMIs, and for more general systems derived from first principles. 

%there is a matrix $S_\psi$ such that $u_k = \psi (y_k)$ implies that
%
%\begin{equation}\label{eq:Spsi} s_\psi(\Delta y_k, \Delta u_k, \Delta d_k)= \begin{pmatrix} \Delta y_k \\ \Delta u_k \\ \Delta d_k\end{pmatrix}^\top S_\psi \begin{pmatrix} \Delta y_k \\ \Delta u_k \\ \Delta d_k\end{pmatrix} \leq 0,\end{equation}

\begin{figure*}[!bh]
	\hrulefill 
	\begin{equation}\label{eq:big LMI}
		\begin{bmatrix} I & A^\top  \\  0 & B_1^\top \\ 0 & B_2^\top \end{bmatrix}\begin{bmatrix} \gamma P & 0 \\ 0 & -P \end{bmatrix}\begin{bmatrix} I & A^\top  \\  0 & B_1^\top \\ 0 & B_2^\top \end{bmatrix}^\top + \begin{bmatrix} C_1 & D_{11} & D_{12} \\ 0 & I & 0 \\ C_2 & D_{21} & D_{22} \\ 0 &0 & I \end{bmatrix}^{\!\top} 
		\begin{bmatrix} \lambda S_\varphi  & 0 &0 \\ 0 & - I & 0 \\  0 & 0 & \mu^2 I  \end{bmatrix} 
		\begin{bmatrix} C_1 & D_{11} & D_{12} \\ 0 & I & 0 \\ C_2 & D_{21} & D_{22} \\ 0 &0 & I \end{bmatrix}  \succeq 0. 
	\end{equation}
\end{figure*}

\begin{lemma}\label{thm: diss into cl} 
	Suppose that there exists a quadratic supply rate $s_{\textup{p}}$, and 
	\begin{itemize}
		\item a positive definite $V_{\textup{p}}$ and scalar $0<\gamma_{\textup{p}}<1$ such that 
		\[ V_{\textup{p}}(\Delta \xi_{k+1}) - \gamma_{\textup{p}} V_{\textup{p}}(\Delta \xi_k) \leq s_{\textup{p}} (\Delta \zeta_k, \Delta \nu_k) \] 
		for all increments of trajectories of \eqref{eq:plant}; and
		\item a positive definite $V_{\textup{o}}$ and scalar $0<\gamma_{\textup{o}}<1$ such that 
		\[  V_{\textup{o}}(\Delta x_{k+1}) - \gamma_{\textup{o}} V_{\textup{o}}(\Delta x_k) \leq - s_{\textup{p}}(\Delta d_k, \Delta z_k), \]
		for all increments of trajectories of the interconnection of \eqref{eq:linear system} and \eqref{eq:general oracles}.
	\end{itemize}
	
	Then the system consisting of \eqref{eq:linear system}, \eqref{eq:general oracles}, \eqref{eq:plant}, $\zeta_k =d_k$, and $\nu_k= z_k$ is exponentially contractive with rate $\max \{\gamma_{\textup{p}},\gamma_{\textup{o}}\}$. 
\end{lemma}
The result follows by using the sum of the storage functions $V_p(\Delta \xi)+V_o(\Delta x)$ as a contraction metric.  

It should be noted that we do not assume either the plant or the optimization algorithm to be \textit{open loop contractive}, that is, contractive with inputs ($d_k$ or $\nu_k$) equal to zero. In fact, if we interpret the storage functions as a measure of \textit{energy}, we allow for exchange of energy between the plant and the algorithm.

\subsection{Robustness against noise}

Consider the interconnection between \eqref{eq:linear system} and \eqref{eq:general oracles} and assume that the signal $d$ is assumed to be a (bounded) disturbance as in Example~\ref{ex:open nesterov 1}. We make the following assumption on the \textit{nominal behavior} of the algorithm.

\begin{assumption}\label{asmp:nominal}
	For a fixed $d_k=0$, the system comprised of \eqref{eq:linear system} and \eqref{eq:general oracles} has a unique stationary point that corresponds to a solution $z^\star$ of the corresponding optimization problem. 
\end{assumption}

Our goal is to quantify how the disturbance $d_k$ affects the resulting output trajectory. Consider the $\ell^2$ norm of the signal $d$, given by $\sum_{k=0}^\infty\norm{{d}_k}_2$. We will consider an incremental form of the $\ell^2$ system gain, for which the following supply rate, denoted $s_\ell$, will play a role:
\[ s_\ell(\Delta z_k, \Delta d_k) = \begin{pmatrix} \Delta z_k \\ \Delta d_k\end{pmatrix}^{\!\top}\!\! \begin{pmatrix} -I & 0 \\ 0 &\mu^2 I \end{pmatrix} \!\!\begin{pmatrix} \Delta z_k \\ \Delta d_k\end{pmatrix}. \]
A system incrementally dissipative with respect to this supply rate is said to have incremental $\ell^2$ gain less than $\mu$. 

%\begin{lemma}[Incremental small gain]
%	\label{lem:l2} 
%	Consider the supply rate $s_\ell(\Delta y_k, \Delta u_k)$ from \eqref{eq:inc small gain} with $\mu\geq 0$. Suppose that the system \eqref{eq:plant} is incrementally dissipative with respect to $s_\ell(\Delta y_k, \Delta u_k)$ and storage function $V$. 
%	
%	If $\sum_{k=0}^\infty \norm{\Delta u_k}^2_2$ is finite, then  
%	\[ \sum_{k=0}^\infty \norm{\Delta y_k}^2_2 \leq V(\Delta x_0) + \mu^2\sum_{k=0}^\infty \norm{\Delta u_k}^2_2. \] 
%\end{lemma} 
%\begin{proof} By definition, we have 
%	\[  V(\Delta x_{k+1}) - V(\Delta x_k) \leq s(\Delta y_k, \Delta u_k), \]
%	Summing these terms for $k=0$ to $k=K$ yields
%	\[ V(\Delta x_K) -V(\Delta x_0) \leq \sum_{k=0}^K (-\norm{\Delta y_k}^2_2 +\mu^2\norm{\Delta u_k}^2_2 ).\] 
%	Note that $V(\Delta x_K)\geq 0$ by definition, and hence we can rearrange the terms to write:
%	\[ \sum_{k=0}^K \norm{\Delta y_k}^2_2  \leq V(\Delta x_0) + \mu^2\sum_{k=0}^K \norm{\Delta u_k}^2_2.\] 
%	Since the limit $K\rightarrow \infty$ of the right hand side is assumed to be finite, we can conclude the lemma.  
%\end{proof} 

In line with Lemma~\ref{lem:LMI}, we can test whether the interconnection between \eqref{eq:linear system} and \eqref{eq:general oracles} has incremental $\ell^2$ gain less than $\mu$. The following theorem gives a handle on the repercussions of this for bounding the effect of disturbances on the output of the system.
\begin{theorem}\label{cor:ell2}
	Consider the interconnection between \eqref{eq:linear system} and \eqref{eq:general oracles}. Assume that the oracles satisfy a bound of the form \eqref{eq:Sphi}. Suppose that the inequality \eqref{eq:big LMI} holds for $P\succ 0$ and $0<\gamma<1$, then:
	\begin{itemize}
		\item The system has incremental $\ell^2$ gain less than $\mu$.
		\item If $d_k =0$, the system is exponentially contractive with rate $\gamma$, towards a fixed point $x^\star$ with output $z^\star$.
		\item Let $z_k$ denote an output trajectory of the closed loop initialized at $x_0$ corresponding to the signal $d_k$. Then for any sequence $d_k$ for which $\sum_{k=0}^\infty\norm{{d}_k}_2$ is finite, we have:
		\[ \sum_{k=0}^\infty \norm{z^\star-z(x_0;d_k)}^2_2 \leq V(x^\star-x_0)+ \mu^2 \sum_{k=0}^\infty \norm{{d}_k}^2_2. \] 
	\end{itemize}
\end{theorem}

In simple terms, under these conditions the \textit{nominal} system is exponentially contractive. Under Assumption~\ref{asmp:nominal}, this has to be towards the solution of the optimization problem. Moreover, the second part of Theorem~\ref{cor:ell2} means that any other output trajectory is bounded by the sum of an initial error term and $\mu^2$ times the $\ell^2$ norm of the disturbance. 

The above result verifies a given nominal contraction rate and bound on the incremental $\ell^2$ gain. Recall that we are interested in the robustness trade-off. That is, the trade-off between the nominal contraction rate and the incremental $\ell^2$ gain of the system. For this, we can make two observations: 
\begin{itemize} 
	\item The scalar $\mu^2$ appears linearly in \eqref{eq:big LMI}. Therefore, given $0<\gamma<1$, minimizing for $\mu>0$ can be reframed as a semi-definite program (SDP) and solved efficiently. 
	\item The scalar $\gamma$ appears bilinearly, multiplied with $P$. Minimizing for $\gamma$ for which the LMI is feasible can be done using bisection. 
\end{itemize} 

Further trade-offs appear when also considering the \textit{overshoot} of the algorithm. Let $x^*$ be a stationary point of the dynamics, assume that and the dynamics are exponentially contractive with contraction metric  $V(\Delta x_k)=(\Delta x_k)^\top P (\Delta x_k)$ with $P\succ 0$ and rate $\gamma$. Let $\lambda_-(P)$ and $\lambda_+(P)$ denote the smallest and largest eigenvalue of $P$ respectively, then 
\[ \norm{x_k-x^*}^2_2 \leq \gamma^k\frac{\lambda_+(P)}{\lambda_-(P)} \norm{x_0-x^*}^2_2.\]
for any trajectory $x_k$ of the algorithm. Hence, we can quantify the maximal overshoot in terms of the eigenvalues of $P$. 

\begin{remark}\label{rem:smallgain}
	A special case of Lemma~\ref{thm: diss into cl} arises if the plant has incremental $\ell^2$ gain less than $\mu_{\textup{p}}$, and the optimization algorithm has incremental $\ell^2$ gain less than $\mu_{\textup{o}}$. Then, if $\mu_{\textup{p}}\mu_{\textup{o}}\leq1$ the closed loop system is non-expansive. 
\end{remark}

\section{Numerical example}
In this section we consider open optimization algorithms with \eqref{eq:linear system} given by
\begin{align*}
& A = \begin{bmatrix}1+\beta &-\beta \\ 1 & 0\end{bmatrix}, & \quad & B_1=B_2= \begin{bmatrix} -\alpha \\ 0 \end{bmatrix}, \\
& C_1=C_2= \begin{bmatrix}1+\eta &-\eta \end{bmatrix}, & \quad & D= 0, \end{align*} 
parametrized by scalars $\alpha,\beta,\eta$. For the nominal system, that is, if $d_k=0$, it was shown in \cite{BvS-LL:21} that this fully characterizes the set of scalar optimization algorithms with state dimension equal to $2$ and an observable eigenvalue equal to $1$. Moreover, there are a number of common choices as special cases, see Table~\ref{table:nominal-algos}.

\begin{table*}[ht]
	\centering
	\begin{tabular}[t]{lccccc}
		\toprule
		Nominal algorithm name & $\alpha$ & $\beta$ & $\eta$ & minimal $\rho$ & minimal $\mu$ \\
		\midrule
		Gradient Descent (GD) & $\frac{2}{L+m}$ & 0 & 0 & 1.0019 & 1.0039  \\
		Gradient Descent (GDL) & $\frac{1}{L}$ & 0 & 0 &  0.3245 & 1.0004  \\
		Gradient Descent (GD10L) & $\frac{1}{10L}$ & 0 & 0 & 0.0176 &  1.0004 \\
		Heavy Ball (HB) \cite{Polyak-book} & $\frac{4}{(\sqrt{L}+\sqrt{m})^2}$ &  $\left(\frac{\sqrt{L}-\sqrt{m}}{\sqrt{L}+\sqrt{m}}\right)^2$ & 0 & 1.0020 & \\
		Fast Gradient (FG) \cite{Nesterov-book} & $\frac{1}{L}$ & $\frac{\sqrt{L}-\sqrt{m}}{\sqrt{L}+\sqrt{m}}$ & $\frac{\sqrt{L}-\sqrt{m}}{\sqrt{L}+\sqrt{m}}$ &0.5880 &  1.7679\\
		Triple Momentum (TM) \cite{BvS-2017} & $\frac{2\sqrt{L}-\sqrt{m}}{L^{3/2}}$ & $\frac{(\sqrt{L}-\sqrt{m})^2}{L+\sqrt{mL}}$ & $\frac{(\sqrt{L}-\sqrt{m})^2}{2L-m+\sqrt{mL}}$ & 0.9892 &    5.6812 \\
		Robust Momentum (RM) \cite{BvS-LL:21} & $\frac{(1-r_i)^2(1+r_i)}{m}$ & $\frac{Lr_i^3}{L-m}$ & $\frac{mr_i^3}{(L-m)(1-r_i)^2(1+r_i)}$ & 0.4727 & 1.2039 \\
		\bottomrule
	\end{tabular}
		\caption{List of different parameter choices, for which the nominal algorithms are recommended/standard tunings of well-known algorithms (cf. \cite{BvS-LL:21}). For RM, we consider on the parameter $r$ in the middle of the interval of feasible values $[1-\sqrt{\frac{m}{L}}, 1-\frac{m}{L}]$. The last two columns are the minimal numerical value of $\rho$ (cf. \ref{eq:strict-inc-passive}) and upper bound for the incremental $\ell^2$ gain $\mu$ (Theorem~\ref{cor:ell2}) with $m=1,L=6$. Note that the latter was infeasible for the Heavy ball method.} 
	\label{table:nominal-algos}
\end{table*}

We will focus on optimizing a $\mu$-strongly convex function $c$ with $L$-Lipschitz gradients. With some over-approximation, we assume that we have access to an oracle $u_k = \varphi(y_k) := \nabla c(y_k)$, for which
\[ s(\Delta y_k, \Delta u_k) = \begin{pmatrix} \Delta y_k \\ \Delta u_k\end{pmatrix}^\top \begin{bmatrix} mL & -\tfrac{m+L}{2} \\ -\tfrac{m+L}{2} & 1 \end{bmatrix} \begin{pmatrix} \Delta y_k \\ \Delta u_k\end{pmatrix} \leq 0. \] 
For the simulations, we will consider $m=1$ and $L=6$. We will investigate the \textit{strict incremental passivity index} for the open optimization algorithms, that is, we are interested in the minimal value of $\rho$ for which the interconnection of \eqref{eq:linear system} with the oracle $u_k = \varphi(y_k)$ is dissipative with respect to the incremental supply rate:
\begin{equation}\label{eq:strict-inc-passive}  s_\textup{ext}(\Delta z_k, \Delta d_k) = \begin{pmatrix} \Delta z_k \\ \Delta d_k\end{pmatrix}^\top \begin{bmatrix} \rho & -1 \\ -1 & \rho \end{bmatrix} \begin{pmatrix} \Delta z_k \\ \Delta d_k\end{pmatrix} \end{equation} 

In Table~\ref{table:nominal-algos} we present the resulting numerical values of $\rho$ obtained by solving an SDP using YALMIP \cite{JL:04} and MOSEK. Additonally, Table~\ref{table:nominal-algos} shows upper bounds on the incremental $\ell^2$ gain found using Theorem~\ref{cor:ell2}.

\begin{figure}
	\includegraphics[width=0.5\textwidth]{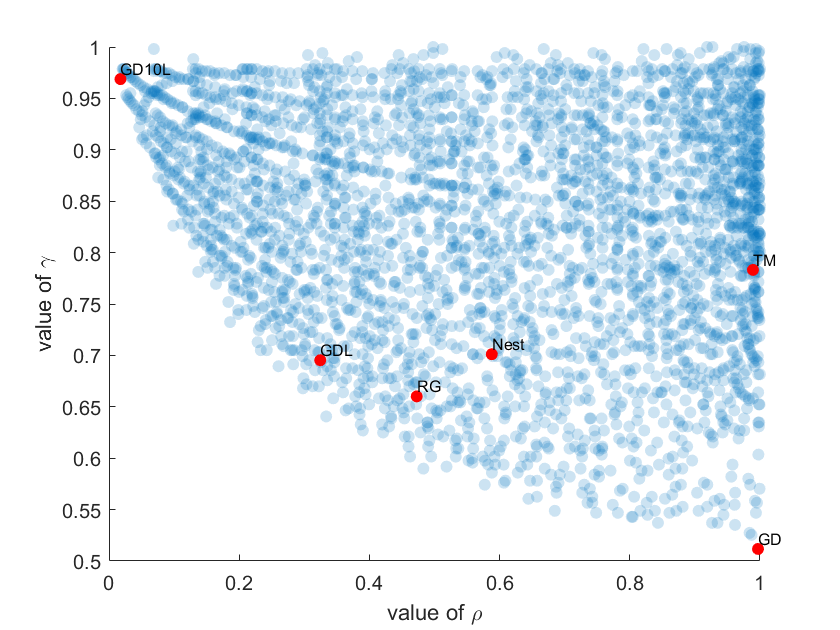}
	\caption{The minimal open-loop contraction rate $\gamma$ plotted against the minimal strict incremental passivity index $\rho$ (cf. \eqref{eq:strict-inc-passive}) for a sweep of 18000 sets of parameters $\alpha,\beta,\eta$. Also shown are a number of special cases labeled as in Table~\ref{table:nominal-algos}.} \label{fig:sweep}
\end{figure}

\begin{figure}
	\includegraphics[width=0.5\textwidth]{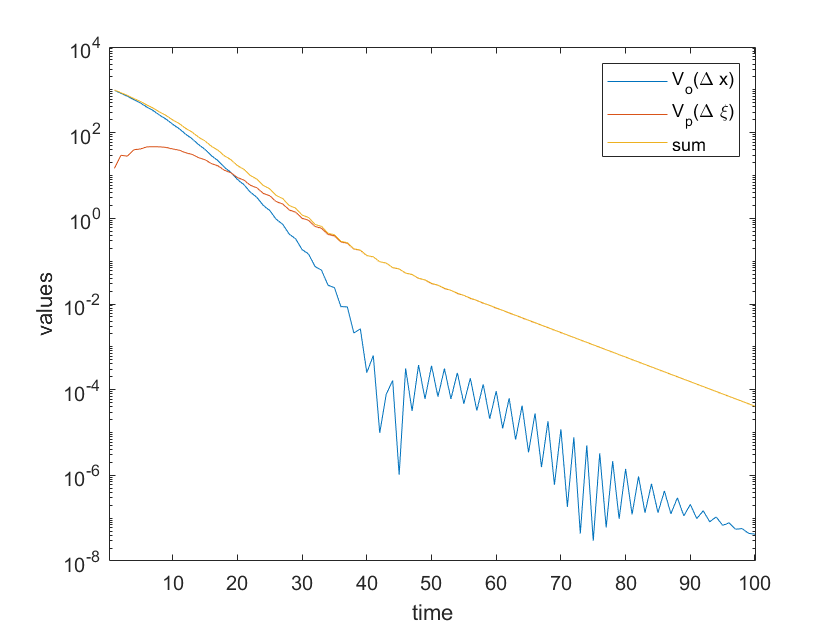}
	\caption{The values of the incremental storage functions of the plant $V_{\textup{p}}$, of the algorithm $V_{\textup{o}}$, and their sum for a randomly initiated pair of trajectories. For the simulation we chose the oracle equal to $u_k = \phi(y_k) = y_k+5\tanh(y_k)+2$. Note that the 'energy' in the plant or the algorithm does not decrease uniformly, that is, there is energy exchange between the plant and the algorithm.} \label{fig:contr}
\end{figure}

In Figure~\ref{fig:sweep} we present the results of a sweep of 18000 sets of parameters $\alpha,\beta,\eta$. Note that there is Pareto front between the contaction rate and the passivity indices, and this is not attained by a standard method.

\subsection*{Closed-loop optimization} 
Now consider the linear plant given by 
\[ \xi_{k+1}= \begin{bmatrix} -0.9 & 0.3 \\ 0.2 & 0.8\end{bmatrix} \xi_k + \begin{bmatrix} 0 \\ 1 \end{bmatrix} v_k  \quad \zeta_k = \begin{bmatrix} 1 & 0 \end{bmatrix} \xi_k + v_k.  \] 
We can verify numerically that the plant has incremental $\ell^2$ gain smaller than $\mu= 3.5001
$. Suppose that we want to use a small gain argument (cf. Remark~\ref{rem:smallgain}) to obtain a nonexpansive closed loop. This requires an algorithm with incremental $\ell^2$ gain smaller than $\tfrac{1}{\mu}$, which  Table~\ref{table:nominal-algos} contains no examples of. 

However, the plant is also (strictly) dissipative with respect to 
\[ s_\textup{p}(\Delta \zeta_k, \Delta \nu_k) = \begin{pmatrix} \Delta \zeta_k \\ \Delta \nu_k\end{pmatrix}^\top \begin{bmatrix} \delta & 1 \\ 1 & \delta \end{bmatrix} \begin{pmatrix} \Delta \zeta_k \\ \Delta \nu_k\end{pmatrix}, \] 
with \textit{antipassivity index} $\rho= -0.5282$. Using Lemma~\ref{thm: diss into cl}, we see that the closed loop will be contractive with GDL, GD10L, and RM. It can be verified that this is indeed the case, and in fact, the system can be shown to not be contractive when interconnected with any of the other choices. In Figure~\ref{fig:contr}, we show the resulting values of the storage functions for the choice of GD10L.

\section{Conclusions}
We considered the analysis of \textit{open optimization algorithms}, in order to analyze robustness and closed loop performance of optimization algorithms in closed loop with dynamical systems. Towards this, we first formalized the use of incremental dissipativity for the analysis of such optimization algorithms, and its links to monotone operator theory and contraction theory. Then, we derived analysis results for optimization-in-the-loop, where we consider the algorithm in interconnection with a plant. Lastly, we provided tests for an open optimization algorithm to have an incremental finite small gain property, and shown its relevance for robustness analysis. These tests take the form of linear matrix inequalities, and can therefore be checked efficiently.

\bibliography{references}

\begin{thebibliography}{10}
\providecommand{\url}[1]{#1}
\csname url@samestyle\endcsname
\providecommand{\newblock}{\relax}
\providecommand{\bibinfo}[2]{#2}
\providecommand{\BIBentrySTDinterwordspacing}{\spaceskip=0pt\relax}
\providecommand{\BIBentryALTinterwordstretchfactor}{4}
\providecommand{\BIBentryALTinterwordspacing}{\spaceskip=\fontdimen2\font plus
\BIBentryALTinterwordstretchfactor\fontdimen3\font minus
  \fontdimen4\font\relax}
\providecommand{\BIBforeignlanguage}[2]{{%
\expandafter\ifx\csname l@#1\endcsname\relax
\typeout{** WARNING: IEEEtran.bst: No hyphenation pattern has been}%
\typeout{** loaded for the language `#1'. Using the pattern for}%
\typeout{** the default language instead.}%
\else
\language=\csname l@#1\endcsname
\fi
#2}}
\providecommand{\BIBdecl}{\relax}
\BIBdecl

\bibitem{JBR-DQM:09}
J.~B. Rawlings and D.~Q. Mayne, \emph{Model predictive control: theory and
  design}.\hskip 1em plus 0.5em minus 0.4em\relax Madison, WI: Nob Hill Pub.
  cop., 2009.

\bibitem{JC-JL-FD:21}
J.~Coulson, J.~Lygeros, and F.~D\"{o}rfler, ``Distributionally robust chance
  constrained data-enabled predictive control,'' \emph{IEEE Transactions on
  Automatic Control}, vol.~67, no.~7, p. 3289, 2021.

\bibitem{OFO20}
M.~Colombino, E.~Dall'Anese, and A.~Bernstein, ``Online optimization as a
  feedback controller: Stability and tracking,'' \emph{IEEE Transactions on
  Control of Network Systems}, vol.~7, no.~1, pp. 422--432, 2020.

\bibitem{HAUSWIRTH2024100941}
A.~Hauswirth, Z.~He, S.~Bolognani, G.~Hug, and F.~D\"{o}rfler, ``Optimization
  algorithms as robust feedback controllers,'' \emph{Annual Reviews in
  Control}, vol.~57, p. 100941, 2024.

\bibitem{Ames17}
A.~D. Ames, X.~Xu, J.~W. Grizzle, and P.~Tabuada, ``Control barrier function
  based quadratic programs for safety critical systems,'' \emph{IEEE
  Transactions on Automatic Control}, vol.~62, no.~8, pp. 3861--3876, 2017.

\bibitem{STOA24}
F.~D\"{o}rfler, Z.~He, G.~Belgioioso, S.~Bolognani, J.~Lygeros, and
  M.~Muehlebach, ``Toward a systems theory of algorithms,'' \emph{IEEE Control
  Systems Letters}, vol.~8, pp. 1198--1210, 2024.

\bibitem{BvS-LL:21}
B.~{Van Scoy} and L.~Lessard, ``The speed-robustness trade-off for first-order
  methods with additive gradient noise,'' 2021, available on arXiv: 2109.05059.

\bibitem{helmke2012optimization}
U.~Helmke and J.~B. Moore, \emph{Optimization and dynamical systems}.\hskip 1em
  plus 0.5em minus 0.4em\relax Springer Science \& Business Media, 2012.

\bibitem{JCW:72a}
J.~C. Willems, ``Dissipative dynamical systems--{P}art {I}: {G}eneral theory,''
  \emph{Archive for Rational Mechanics and Analysis}, vol.~45, no.~5, pp.
  321--351, 1972.

\bibitem{JCW:72b}
------, ``Dissipative dynamical systems--{P}art {II}: {L}inear systems with
  quadratic supply rates,'' \emph{Archive for Rational Mechanics and Analysis},
  vol.~45, no.~5, pp. 352--393, 1972.

\bibitem{Lessard-22}
L.~Lessard, ``The analysis of optimization algorithms: A dissipativity
  approach,'' \emph{IEEE Control Systems Magazine}, vol.~42, no.~3, pp. 58--72,
  2022.

\bibitem{lurie1944theory}
A.~I. Lur'e and V.~N. Postnikov, ``{K teorii ustoichivosti reguliruemykh sistem
  [A contribution to the theory of the stability of control systems]},''
  \emph{{Prikladnaya Matematika i Mekhanika}}, no.~8, pp. 246--248, 1944.

\bibitem{SP-19}
J.~W. Simpson-Porco, ``Equilibrium-independent dissipativity with quadratic
  supply rates,'' \emph{IEEE Transactions on Automatic Control}, vol.~64,
  no.~4, pp. 1440--1455, 2019.

\bibitem{Lessard16}
L.~Lessard, B.~Recht, and A.~Packard, ``Analysis and design of optimization
  algorithms via {Integral Quadratic Constraints},'' \emph{SIAM Journal on
  Optimization}, vol.~26, no.~1, pp. 57--95, 2016.

\bibitem{Scherer-21}
C.~Scherer and C.~Ebenbauer, ``Convex synthesis of accelerated gradient
  algorithms,'' \emph{SIAM Journal on Control and Optimization}, vol.~59,
  no.~6, pp. 4615--4645, 2021.

\bibitem{scherer23}
C.~W. Scherer, C.~Ebenbauer, and T.~Holicki, ``Optimization algorithm synthesis
  based on integral quadratic constraints: A tutorial,'' in \emph{CDC}, 2023,
  pp. 2995--3002.

\bibitem{Zames60}
G.~Zames, \emph{Nonlinear Operators of System Analysis}, 1960, {Ph.D.} thesis,
  Massachusetts Institute of Technology.

\bibitem{LOHMILLER1998683}
W.~Lohmiller and J.~J.~E. Slotine, ``On contraction analysis for non-linear
  systems,'' \emph{Automatica}, vol.~34, no.~6, pp. 683--696, 1998.

\bibitem{FB-CTDS}
F.~Bullo, \emph{Contraction Theory for Dynamical Systems}, {1.2}~ed.\hskip 1em
  plus 0.5em minus 0.4em\relax Kindle Direct Publishing, 2024.

\bibitem{SEPULCHRE2022290}
R.~Sepulchre, T.~Chaffey, and F.~Forni, ``On the incremental form of
  dissipativity,'' \emph{IFAC-PapersOnLine}, vol.~55, no.~30, pp. 290--294,
  2022, {MTNS}.

\bibitem{bauschke2011}
H.~H. Bauschke and P.~L. Combettes, \emph{Convex analysis and monotone operator
  theory in Hilbert spaces}.\hskip 1em plus 0.5em minus 0.4em\relax Springer,
  2011.

\bibitem{ryu2016primer}
E.~K. Ryu and S.~Boyd, ``Primer on monotone operator methods,'' \emph{Appl.
  comput. math}, vol.~15, no.~1, pp. 3--43, 2016.

\bibitem{Nesterov-book}
Y.~Nesterov, \emph{Lectures on convex optimization, Second edition}.\hskip 1em
  plus 0.5em minus 0.4em\relax Springer, vol. 137.

\bibitem{Azkathesis23}
A.~M. Burohman, \emph{From data to reduced-order models of complex dynamical
  systems}, 2023, {Ph.D.} thesis, University of Groningen.

\bibitem{Polyak-book}
B.~Polyak, \emph{Introduction to Optimization}, ser. Translations series in
  mathematics and engineering.\hskip 1em plus 0.5em minus 0.4em\relax
  Optimization Software, Publications Division, 1987.

\bibitem{BvS-2017}
B.~{Van Scoy}, R.~A. Freeman, and K.~M. Lynch, ``The fastest known globally
  convergent first-order method for minimizing strongly convex functions,''
  \emph{IEEE Control Systems Letters}, vol.~2, no.~1, pp. 49--54, 2017.

\bibitem{JL:04}
J.~{Lofberg}, ``{YALMIP}: a toolbox for modeling and optimization in
  {MATLAB},'' 2004, pp. 284--289.

\end{thebibliography}

\bibliographystyle{IEEEtran}

\end{document}